\documentclass[11pt]{article}
\usepackage{latexsym}
\usepackage{amsfonts}
\usepackage{amscd}
\usepackage{array}
\usepackage{amsmath}
\usepackage{epsfig}
\usepackage{amssymb}
\usepackage{amsthm}
\newtheorem{theorem}{Theorem}[section]
\newtheorem{lemma}[theorem]{Lemma}
\newtheorem{proposition}[theorem]{Proposition}
\newtheorem{conjecture}[theorem]{Conjecture}
\newtheorem{corollary}[theorem]{Corollary}

\theoremstyle{definition}
\newtheorem{definition}[theorem]{Definition}
\newtheorem{example}[theorem]{Example}

\newtheorem{remark}[theorem]{Remark}
\newtheorem{question}[theorem]{Question}

\begin{document}

\title{Gorenstein polytopes and their stringy $E$-functions}

\author{Benjamin Nill and Jan Schepers\thanks{Postdoctoral Fellow of the Research Foundation - Flanders (FWO).}}
\date{}

\maketitle



\begin{abstract}
\noindent Inspired by ideas from algebraic geometry, Batyrev and the first named author have introduced the stringy $E$-function of a Gorenstein polytope. We prove that this a priori rational function is actually a polynomial, which is part of a conjecture of Batyrev and the first named author. The proof relies on a comparison result for the lattice point structure of a Gorenstein polytope $P$, a face $F$ of $P$ and the face of the dual Gorenstein polytope corresponding to $F$. In addition, we study joins of Gorenstein polytopes and introduce the notion of an irreducible Gorenstein polytope. We show how these concepts relate to the decomposition of nef-partitions.
\end{abstract}




\section{Introduction}

In this paper we investigate the stringy $E$-function of a Gorenstein polytope, as defined by Batyrev and the first named author in \cite{BatyrevNill}. We start by sketching the geometric ideas behind its combinatorial definition. 

In \cite{Batyrev1} Batyrev introduced \textit{reflexive polytopes} together with the notion of duality between them (see Definition~\ref{reflexive}). He conjectured that nondegenerate Calabi-Yau hypersurfaces $X_{\Delta}$ and $X_{\Delta^{\times}}$ in the projective toric varieties associated to dual reflexive polytopes $\Delta$, respectively $\Delta^{\times}$, should have mirror dual stringy Hodge numbers, in the sense that
\begin{equation*}\label{mirrorsymm}
h_{st}^{p,q}(X_{\Delta})  =  h_{st}^{n-p,q}(X_{\Delta^{\times}}) \quad \text{ for } p,q \in \{0,\ldots, n\} , \tag{*}
\end{equation*}
where $n$ is the dimension of $X_{\Delta}$ and $X_{\Delta^{\times}}$. There is a need for \textit{stringy Hodge numbers} here (as defined in generality in \cite{Batyrev}), since the involved Calabi-Yau varieties may be singular. 

This idea was conjecturally generalised by Borisov \cite{Borisov} to Calabi-Yau complete intersections defined by a \textit{nef-partition} of a reflexive polytope $\Delta$ (see Definition~\ref{nefpartition}). Together, Batyrev and Borisov could actually settle this version of mirror symmetry \cite[Thm.\ 4.15]{BatyrevBorisov}. The main difficulty in their proof was to find an explicit formula for the \textit{stringy $E$-function} of a generic Calabi-Yau complete intersection $Y$.
The stringy $E$-function is the generating function of the stringy Hodge numbers of $Y$:    
\[ E_{st}(Y;u,v) := \sum_{p,q} (-1)^{p+q}\, h_{st}^{p,q}(Y)\, u^p v^q \quad \in \mathbb{Z}[u,v]. \]

A nef-partition gives a Minkowski sum decomposition $\Delta = \Delta_1 + \cdots + \Delta_r$. Batyrev and Borisov consider the \textit{Cayley cone} $\sigma$ associated to $\Delta_1,\ldots, \Delta_r$ (see Definition~\ref{Cayley}). If $\dim \Delta =e$ then $\sigma$ is a \textit{Gorenstein cone} of dimension $e+r$ with a \textit{Gorenstein polytope $\tilde{\Delta}$ of index $r$} as support polytope (see Definitions~\ref{GorensteinPolytope} and~\ref{GorensteinCone}). The results of Batyrev and Borisov are summarised by the following formula which can be deduced from \cite[Thm.\ 4.14]{BatyrevBorisov} as in \cite[Thm.\ 7.2]{BorisovMavlyutov}:
\[ E_{st}(Y;u,v) = \frac{1}{(uv)^r} \, \sum_{\emptyset \leq F \leq \tilde{\Delta}} (-u)^{\dim F + 1} \, \widetilde{S}(F,u^{-1}v) \, \widetilde{S}(F^{*},uv), \]
where we sum over all faces $F$ of the Gorenstein polytope $\tilde{\Delta}$, where $F^{*}$ denotes the face of the dual Gorenstein polytope corresponding to $F$, and where the polynomial $\widetilde{S}(P,t) \in \mathbb{Z}[t]$ is an invariant of a lattice polytope $P$ combining information about the lattice points in multiples of $P$ and information about the face poset of $P$ (see Definition~\ref{Stilde}).

The idea of Batyrev and the first named author was to use this last formula as definition for the stringy $E$-function of an arbitrary Gorenstein polytope (see Definition~\ref{stringyE}). In the above setting, the stringy $E$-function is a polynomial in $u,v$ of degree $2(e-r) = 2(\dim\tilde{\Delta} +1 -2r)$, but this is a priori not guaranteed for arbitrary Gorenstein polytopes. It is the content of the following conjecture \cite[Conj.\ 4.10]{BatyrevNill}:

\begin{conjecture}\label{conjecture}
Let $P$ be a Gorenstein polytope of index $r$ and dimension $d$. Then $E_{st}(P;u,v)$ is a polynomial of degree $2(d+1 - 2r)$ or it is $0$ (in particular, it is $0$ if $d+1 < 2r$).  
\end{conjecture}

In Section~\ref{sectionproof} we prove that the stringy $E$-function is indeed a polynomial. An essential ingredient for the proof is a comparison theorem for the lattice point structure of a Gorenstein polytope $P$, a face $F$ of $P$ and the dual face $F^*$ (see Theorem~\ref{maintheorem}). In Section~\ref{joins} we study several notions of \textit{joins} of Gorenstein polytopes. This leads in a natural way to the definition of an \textit{irreducible} Gorenstein polytope (see Definition~\ref{Irreducible}). In Section~\ref{SectionNef} we show that this notion relates well to the notion of irreducible nef-partitions from \cite{BatyrevBorisovCompleteIntersections}. Finally, in Section~\ref{Ideas} we present several questions which might be interesting starting points for further investigation.\\

\noindent \textbf{Acknowledgements.} The computer program \textsc{Normaliz} 2.0 by Bruns and Ichim \cite{normaliz} was very useful for computing examples and in that way it was a source of inspiration. The first named author would like to thank the MSRI Berkeley for a Postdoctoral Fellowship during the Tropical Geometry program 2009 and Janko B\"ohm for many discussions about the stringy $E$-functions of Gorenstein polytopes. The second named author thanks the Institut des Hautes \'Etudes Scientifiques (IH\'ES), where part of this work was done, for its hospitality.

\section{Definitions and preliminaries}\label{definitions}

We recall all definitions that will be used later. We mainly follow \cite{BatyrevNill}.

\subsection{Gorenstein polytopes and cones}\label{subsectionGorenstein}

Let $M\cong \mathbb{Z}^d$ be a lattice of rank $d$ and $N=\text{Hom}_{\mathbb{Z}}(M,\mathbb{Z})$ the dual lattice. Denote by $\langle n,m \rangle\in \mathbb{Z}$ the natural pairing between elements $n\in N$ and $m\in M$. We denote the dual real vector spaces $M\otimes \mathbb{R}$ and $N\otimes \mathbb{R}$ by $M_{\mathbb{R}}$ and $N_{\mathbb{R}}$. Of course, the pairing $\langle\cdot,\cdot\rangle$ extends to a pairing $\langle\cdot,\cdot\rangle: M_{\mathbb{R}}\times N_{\mathbb{R}}\to \mathbb{R}$. 

By a \textit{lattice polytope} $P$ in $M_{\mathbb{R}}$ we mean the convex hull of finitely many points of $M$. The \textit{dimension} of $P$ is the dimension of the smallest affine subspace of $M_{\mathbb{R}}$ containing $P$. We consider $\emptyset$ as a lattice polytope of dimension $-1$. A \textit{face} of $P$ is any intersection of a hyperplane $H$ in $M_{\mathbb{R}}$ with $P$, such that $P$ is completely contained in one of the two closed half spaces determined by $H$ (it is convenient to consider $\emptyset$ and $P$ also as faces, but they are called \textit{improper} faces). We write $F\leq P$ if $F$ is a face of $P$. A \textit{facet} $F$ of $P$ is a face of codimension 1. If $P$ has dimension $d$, there is a unique hyperplane cutting out a facet $F$. It is called the \textit{supporting} hyperplane. Let $m\in M,n\in N$ and assume that $n$ is a primitive lattice point (i.e.\ the segment between the origin and $n$ in $N_{\mathbb{R}}$ does not contain any other lattice point). The \textit{integral distance} between $m$ and the hyperplane $H_{n,\delta} = \{ x\in M_{\mathbb{R}}  \,|\, \langle n , x \rangle = \delta  \}$, where $\delta\in \mathbb{Z}$, is the nonnegative integer $| \langle n , m \rangle - \delta |$. 

\begin{definition}\label{reflexive}
Let $P$ be a lattice polytope of dimension $d$ in $M_{\mathbb{R}}$. It is called a \textit{reflexive polytope} if $0$ is in the interior of $P$ and if $0$ has integral distance $1$ to all supporting hyperplanes of facets of $P$.
\end{definition}

In this case $0$ is the only interior lattice point. We could also have used the following criterion as definition.

\begin{proposition}
Let $P$ be a lattice polytope in $M_{\mathbb{R}}$ of dimension $d$ having $0$ in its interior. Then $P$ is reflexive if and only if the dual set 
\[ P^{\times} = \{ y\in N_{\mathbb{R}} \,|\, \langle y,x \rangle \geq -1 \text{ for all } x\in P \} \] is a lattice polytope in $N_{\mathbb{R}}$.
\end{proposition}

In that case $P^{\times}$ is also a reflexive polytope, called the \textit{dual polytope} of $P$. This duality induces an inclusion reversing 1-1-correspondence between the faces of $P$ and of $P^{\times}$ by
\[ F \leq P \mapsto F^{*}= \{ y\in P^{\times} \,|\, \langle y, x \rangle = -1 \text{ for all } x\in F \} \leq P^{\times}. \] Note that $\dim F + \dim F^{*} = d-1$. 

In the literature one denotes the dual reflexive polytope usually by $P^*$, but we decided to use the notation $P^{\times}$ since, by the above correspondence, $P^*$ also means the empty face of the dual polytope.

Next we define Gorenstein polytopes.

\begin{definition}\label{GorensteinPolytope}
Let $P$ be a lattice polytope of dimension $d$ in $M_{\mathbb{R}}$. Then $P$ is called a \textit{Gorenstein polytope of index} $r$ ($r\in \mathbb{Z}_{>0}$) if $rP$ contains a unique interior lattice point $m$ and if $rP-m$ is a reflexive polytope.
\end{definition}

\begin{example}
The standard simplex in $\mathbb{R}^d$ is a Gorenstein polytope of index $d+1$ for the lattice $\mathbb{Z}^d$. The standard unit cube is a Gorenstein polytope of index 2.
\end{example}

To define the dual of a Gorenstein polytope, we first need to introduce Gorenstein cones. Let $M$ and $N$ be dual lattices of rank $d$ with pairing $\langle \cdot ,\cdot \rangle $ as above. By a \textit{cone} $C$ in $M_{\mathbb{R}}$ we will always mean a pointed rational polyhedral cone, with vertex at 0. This means that the only linear subspace contained in $C$ is $\{0\}$ and that there are finitely many points $p_1,\ldots ,p_n$ in $M$ such that 
\[ C  = \{ \alpha_1 p_1 + \cdots + \alpha_n p_n \, | \, \alpha_i \in \mathbb{R}_{\geq 0}\text{ for all } i  \}.\]
We say that $C$ is \textit{generated} by $p_1,\ldots , p_n$. The dimension and faces of a cone are defined as for a polytope but now the empty set is not considered to be a face. We will always assume that a cone has maximal dimension $d$. The dual
\[ C^{\vee} := \{ y\in N_{\mathbb{R}} \,|\, \langle y,x \rangle \geq 0 \text{ for all } x\in C \} \]
is a cone again. There is again a 1-1-correspondence 
\[F\mapsto F^{\vee}=\{ y\in C^{\vee} \,|\, \langle y, x \rangle = 0 \text{ for all } x\in F \}   \] between faces of $C$ and $C^{\vee}$, now with $\dim F + \dim F^{\vee}=d$. 

\begin{definition}\label{GorensteinCone}
A cone $C$ in $M_{\mathbb{R}}$ is called a \textit{Gorenstein cone} if there exists a lattice point $n\in N$ such that $C$ is generated by finitely many lattice points in the affine hyperplane $H_{n,1} = \{x\in M_{\mathbb{R}} \,|\,  \langle n,x \rangle = 1\}$. The lattice polytope $C\cap H_{n,1}$ is called the \textit{support polytope} of $C$.
\end{definition}

One can see that the lattice point $n$ lies in the interior of $C^{\vee}$ and that it is uniquely determined by $C$. We denote it by $n_C$. The $k$\textit{-th slice} $C_{(k)}$ of $C$ is the lattice polytope $C \cap \{x\in M_{\mathbb{R}} \,|\,  \langle n_C ,x \rangle = k\}$.

Starting from a lattice $M'$ of rank $e$ and a lattice polytope $P$ in $M'_{\mathbb{R}}$ of dimension $e$, we consider the lattice $M'\oplus \mathbb{Z}$ and its corresponding vector space $M'_{\mathbb{R}}\oplus \mathbb{R}$. We can define the Gorenstein cone $C_P$ \textit{associated to $P$} in $M'_{\mathbb{R}} \oplus \mathbb{R}$ by taking the cone generated by the vertices of $\overline{P} = P \times \{1\}$. 

\begin{definition}\label{reflexiveGorensteincone}
A Gorenstein cone $C$ in $M_{\mathbb{R}}$ is called \textit{reflexive} if $C^{\vee}$ is a Gorenstein cone as well. We denote the unique interior lattice point of $C$ that determines the support polytope of $C^{\vee}$ by $m_{C^{\vee}}$. The positive integer $r=\langle n_C , m_{C^{\vee}} \rangle$ is called the \textit{index} of $C$. 
\end{definition}

If $C$ is a reflexive Gorenstein cone of index $r$, then this is also true for $C^{\vee}$. One uses the following proposition to define the dual of a Gorenstein polytope \cite[Prop.\ 2.11]{BatyrevBorisov2}. 

\begin{proposition}
Let $C$ be a Gorenstein cone in $M_{\mathbb{R}}$. Then the following are equivalent:
\begin{itemize}
\item[(1)] $C$ is reflexive of index $r$,
\item[(2)] $C_{(r)}$ is a reflexive polytope with $m_{C^{\vee}}$ as its unique interior point,
\item[(3)] the support polytope $C_{(1)}$ of $C$ is a Gorenstein polytope of index $r$.
\end{itemize}
\end{proposition}

\begin{definition}
Let $P$ be a $d$-dimensional Gorenstein polytope of index $r$ in $M_{\mathbb{R}}$. The \textit{dual Gorenstein polytope} $P^{\times}$ is the support polytope of the dual $C_P^{\vee}$ of the cone $C_P$ associated to $P$ in $M_{\mathbb{R}}\oplus \mathbb{R}$.
\end{definition}

The duality between the sets of faces of $C_P$ and $C_P^{\vee}$ induces a duality between the sets of faces of $P$ and $P^{\times}$. We denote the face of $P^{\times}$ dual to a face $F$ of $P$ by $F^*$. Then $\dim F + \dim F^{*} = d-1$. Note that we recover the duality for reflexive polytopes (being Gorenstein polytopes of index $1$).

\subsection{Polynomial invariants of lattice polytopes}

Let $M$ be a lattice and let $P$ be a lattice polytope in $M_{\mathbb{R}}$ of dimension $d\geq 0$. Denote the number of lattice points in $mP$ for $m\in \mathbb{Z}_{>0}$ by $f_P(m)$. Then it is well known (see \cite{Ehrhart,Stanley80}) that one can write the \textit{Ehrhart series} $1+\sum_{m\geq 1} f_P(m)\, t^m$ in the form
\[ \frac{h^*_P(t)}{(1-t)^{d+1}},\]
where $h_P^*(t) \in \mathbb{Z}[t]$ is a polynomial with nonnegative coefficients and with degree $s\leq d$. One calls $s$ the \textit{degree} of $P$. The positive integer $d+1-s$ is called the \textit{codegree} of $P$. We denote them by deg\,$P$ and codeg\,$P$ respectively. It follows from Ehrhart reciprocity that the codegree of $P$ is the smallest nonnegative integer $k$ for which $kP$ has a lattice point in its relative interior, and that the leading coefficient of $h^*_P(t)$ is the number of interior points in $(d+1-s)P$. Note that the index of a Gorenstein polytope is simply its codegree. We remark that one defines $h_{\emptyset}^*(t) := 1$. Then $\deg \emptyset = \textnormal{codeg}\,\emptyset = 0$.

Hibi characterized Gorenstein polytopes as follows, see \cite{Hibi}. In view of this result, we consider $\emptyset$ as a Gorenstein polytope as well.

\begin{theorem}\label{resultHibi}
A lattice polytope $P$ is Gorenstein if and only if its $h^*$-po\-ly\-no\-mial is palindromic, i.e.\ $h^*_P(t) = t^{\deg P} h^*_P(t^{-1})$. In particular, the leading coefficient equals $1$.
\end{theorem}

Two lattice polytopes $P\subset M_{\mathbb{R}}$ and $Q\subset M'_{\mathbb{R}}$ are called \textit{isomorphic} if there exists an invertible affine transformation $A$ of $M_{\mathbb{R}}$ with $A(M)=M$ and an isomorphism $B: M'\to M$ of lattices such that $A(P)=B_{\mathbb{R}}(Q)$, where $B_{\mathbb{R}}$ denotes the induced isomorphism $M'_{\mathbb{R}} \to M_{\mathbb{R}}$. Then $h_P^*(t)=h_Q^*(t)$. 


It is well known \cite[p.122]{Stanley97} that the set of faces of a lattice polytope $P$ is an Eulerian poset with rank function $\rho(F)=\dim F +1$.

\begin{definition}
Let $\mathcal{P}$ be an Eulerian poset of rank $e$, with rank function $\rho$, minimal element $\hat{0}$ and maximal element $\hat{1}$. Define 
\[ g(\mathcal{P},t), h(\mathcal{P},t) \in \mathbb{Z}[t] \]
by the recursive rules
\[ g(\mathcal{P},t)=h(\mathcal{P},t)=1 \text{ if } e=0,\]
\[ h(\mathcal{P},t) = \sum_{\hat{0}< x \leq \hat{1}} (t-1)^{\rho(x)-1} g([x,\hat{1}],t) \text{ if } e>0, \]
\[ g(\mathcal{P},t) = \tau_{< e/2} \bigl((1-t)\, h(\mathcal{P},t)\bigr) \text{ if } e>0, \]
where $\tau_{<r}: \mathbb{Z}[t] \to \mathbb{Z}[t] $ is the
truncation operator defined by
\[ \tau_{<r}\biggl(\sum_i a_it^i \biggr) = \sum_{i<r} a_it^i.\]
\end{definition} 

Note that for $e>0$, $\deg h(\mathcal{P},t) =e-1$ and $\deg g(\mathcal{P},t) \leq (e-1)/2$. 

\begin{remark}\label{remarkg}
The $g$-polynomial has the following properties.
\begin{itemize}
\item[(1)] For a product $\mathcal{P}\times \mathcal{Q}$ of posets (with order defined by $(x,y)\leq (x',y')$ if and only if $x\leq x'$ in $\mathcal{P}$ and $y\leq y'$ in $\mathcal{Q}$), we have that $g(\mathcal{P}\times \mathcal{Q},t ) = g(\mathcal{P},t)\, g(\mathcal{Q},t)$.
\item[(2)] If $e> 0$ then $\displaystyle{\sum_{\hat{0} \leq x\leq \hat{1}}}
g([\hat{0},x],t)\, g([x,\hat{1}]^*,t) (-1)^{\rho(\hat{1})-\rho(x)}
=0,$

and also $\displaystyle{\sum_{\hat{0} \leq x\leq \hat{1}}}
(-1)^{\rho(x)-\rho(\hat{0})} g([\hat{0},x]^*,t)\, g([x,\hat{1}],t)
=0$.\\
This is called Stanley's convolution property,
see \cite[Cor.\ 8.3]{Stanley92}.
\end{itemize}
\end{remark}

For face posets of lattice polytopes we have the following beautiful result (a translation of the Hard Lefschetz property of intersection cohomology for projective toric varieties), see \cite[Cor.\ 3.2]{Stanley87}.

\begin{theorem}
Let $P$ be a lattice polytope and let $\mathcal{P}$ be its poset of faces. Then $g(\mathcal{P},t)$ has nonnegative coefficients.
\end{theorem}

Borisov and Mavlyutov combined the $h^*$- and the $g$-polynomial in the following definition \cite[Def.\ 5.3]{BorisovMavlyutov}.

\begin{definition}\label{Stilde}
Let $P$ be a lattice polytope. Define the polynomial $\widetilde{S}(P,t)\in \mathbb{Z}[t]$ by 
\[ \widetilde{S}(P,t) = \sum_{\emptyset \leq F \leq P} (-1)^{\dim P - \dim F} h_F^*(t)\, g([F,P],t),\]
where we sum over all faces $F$ of $P$ and where $[F,P]$ denotes the interval in the Eulerian poset of faces of $P$.
\end{definition}

\begin{remark}\label{remarkStilde}
The polynomial $\widetilde{S}(P,t)$ has many nice properties.
\begin{itemize}
\item[(1)] $\widetilde{S}(\emptyset,t) = 1$ and $\widetilde{S}(P,t) = 0$ if $\dim P =0$. In fact, if $\dim P \geq 0$, then $\widetilde{S}(P,0)=0$ since a $h^*$- and a $g$-polynomial have constant coefficient $1$ and since $P$ has an equal number of even- and odd-dimensional faces.
\item[(2)] The coefficients of $\widetilde{S}(P,t)$ are nonnegative. This is again proved using algebraic geometry \cite[Prop.\ 5.5]{BorisovMavlyutov}.
\item[(3)] One has the reciprocity law $\widetilde{S}(P,t) = t^{\dim P +1}\, \widetilde{S}(P,t^{-1})$, see \cite[Rem.\ 5.4]{BorisovMavlyutov}.
\end{itemize}
\end{remark}





The following formula follows easily from Stanley's convolution property. See \cite[Prop.\ 2.9]{Schepers} for a slightly more general version.

\begin{proposition}
For any lattice polytope $P$ we have
\[h_P^*(t) = \widetilde{S}(P,t) + \sum_{\emptyset \leq F < P} \widetilde{S}(F,t)\,g([F,P]^*,t).\]
\end{proposition}

Recall that all polynomials in the above formula have nonnegative integer coefficients. 

By the \textit{subdegree} of a nonzero univariate polynomial $f$ we mean the degree of the lowest degree nonzero term. We denote it by subdeg\,$f$. If for a lattice polytope $P$ the polynomial $\widetilde{S}(P,t)$ is nonzero, then subdeg\,$\widetilde{S}(P,t) = \dim P+1-\deg \widetilde{S}(P,t)$ by Remark~\ref{remarkStilde} (3). The following corollary is immediate.

\begin{corollary}\label{corollaryStilde} 
Let $P$ be a lattice polytope. Then
\begin{itemize}
\item[(1)] $\widetilde{S}(P,t)\leq h_P^*(t)$ (i.e.\ the inequality holds coefficientwise),
\item[(2)] if $\widetilde{S}(P,t) \neq 0$ then $\deg \widetilde{S}(P,t) \leq \deg P$ and \textnormal{codeg}\,$P \leq $ \textnormal{subdeg}\,$\widetilde{S}(P,t)$.
\end{itemize} 
In particular, if $\deg P < \textnormal{codeg}\,P$ then $\widetilde{S}(P,t) = 0$.
\end{corollary}

It can well happen that $\widetilde{S}(P,t)= 0$ while $\deg P \geq \textnormal{codeg}\,P$. For instance, if $P$ is a lattice pyramid over a facet $F$ then $\widetilde{S}(P,t) = 0$, while $h_P^*(t) = h_F^*(t)$, see \cite[Lemma 4.5]{BatyrevNill}. 

\section{The stringy $E$-function of a Gorenstein polytope}\label{sectionproof}

In this section we prove part of the conjecture of Batyrev and the first named author: the stringy $E$-function of a Gorenstein polytope is a polynomial. We start by discussing the definition of the stringy $E$-function.

\begin{definition}\label{stringyE} \cite[Def.\ 4.8]{BatyrevNill}
Let $P$ be a $d$-dimensional Gorenstein polytope of index $r$. The \textit{stringy $E$-function} is defined by the formula
\[ E_{st}(P;u,v) := \frac{1}{(uv)^r} \, \sum_{\emptyset \leq F \leq P} (-u)^{\dim F + 1}\,\widetilde{S}(F,u^{-1}v)\, \widetilde{S}(F^{*},uv)\quad \in \mathbb{Q}(u,v),\]
where $F^{*}$ is the face of $P^{\times}$ corresponding to $F$.
\end{definition}

As explained in the introduction, this looks like a generating function of stringy Hodge numbers of Calabi-Yau varieties of dimension $d+1 - 2r$. We use plural here, since it is not hard to give examples of polynomial stringy $E$-functions with an arbitrary power of $2$ as leading coefficient and this coefficient would in the geometric setting be precisely the number of irreducible components \cite[Rem.\ 4.11 and 4.22]{BatyrevNill}. See also Question~\ref{leading}. The number $d+1-2r$ is called the \textit{Calabi-Yau dimension} of $P$ and is denoted by CYdim\,$P$. Taking a lattice pyramid over $P$ decreases the Calabi-Yau dimension by 1 and so it can be arbitrarily negative \cite[Rem.\ 4.9 (4)]{BatyrevNill}. We also have
\begin{equation*}\label{CYdim}
\textnormal{CYdim}\,P = \deg P - \textnormal{codeg}\,P. 
\end{equation*}
The Calabi-Yau dimension of $P$ is equal to the Calabi-Yau dimension of its dual $P^{\times}$.

The properties of the stringy $E$-function are summarised in the following proposition \cite[Rem.\ 4.9 (1) and (2)]{BatyrevNill}. They can be proved by using the properties of the $\widetilde{S}$-polynomial (Remark~\ref{remarkStilde}).

\begin{proposition}\label{stringyproperties}
Let $P$ be a Gorenstein polytope of Calabi-Yau dimension $n$. Then
\begin{itemize}
\item[(1)] $E_{st}(P;u,v) = E_{st}(P;v,u)$,
\item[(2)] $(uv)^n\, E_{st}(P;u^{-1},v^{-1}) = E_{st}(P; u,v)$ (Poincar\'e duality),
\item[(3)] $E_{st}(P;u,v) = (-u)^n\, E_{st}(P^{\times};u^{-1},v)$ (mirror symmetry law, see (\ref{mirrorsymm})).
\end{itemize}
\end{proposition}

The precise statement of the conjecture of Batyrev and the first named author is as follows \cite[Conj.\ 4.10]{BatyrevNill}. We have added point (4). 

\begin{conjecture}\label{conjectureBatyrevNill}
Let $P$ be a Gorenstein polytope of Calabi-Yau dimension $n$. 
\begin{itemize}
\item[(1)] Then $E_{st}(P;u,v)$ is a polynomial,
\item[(2)] it is of degree $2n$ or it is $0$ (in particular, it is $0$ if $n<0$),
\item[(3)] for $n\geq 1$ one has $E_{st}(P;u,0) = (-u)^n\, E_{st}(P;u^{-1},0)$,
\item[(4)] $\frac{d}{du}E_{st}(P;u,1)|_{u=1} = \frac{n}{2}\, E_{st}(P;1,1)$,
\item[(5)] and $\frac{d^2}{du^2}E_{st}(P;u,1)|_{u=1} = \frac{n(3n-5)}{12}\, E_{st}(P;1,1)$. 
\end{itemize}
\end{conjecture}

\begin{remark}
${}$
\begin{itemize}
\item[(1)] If $E_{st}(P;u,v)$ is a polynomial, then it follows from Proposition~\ref{stringyproperties} and Remark~\ref{remarkStilde} that it is of the form
\[ E_{st}(P;u,v) = \sum_{p,q=0}^n (-1)^{p+q} \, h_{st}^{p,q}(P) \, u^pv^q\]
where the $h_{st}^{p,q}(P)$ are \textit{nonnegative integers}. We mention that there is a fascinating conjecture of Batyrev about the nonnegativity of the coefficients of polynomial stringy $E$-functions in algebraic geometry \cite[Conj.\ 3.10]{Batyrev}.
\item[(2)] The last three statements are inspired by algebraic geometry: a generating function of stringy Hodge numbers of Calabi-Yau varieties should satisfy (3) by analogy with Serre duality for smooth varieties with trivial canonical class. For the ideas behind (4) and (5) we refer to \cite{BatyrevVirasoro}. We have added (4) to the conjecture since it is required after combining (5) and the multiplicative behaviour with respect to $\mathbb{Z}$-joins (see Section~\ref{joins} and in particular Proposition~\ref{stringyjoin}). 
\item[(3)] It was noted in \cite[Cor.\ 4.18]{BatyrevNill} that constructions as in the introduction can be used to prove the conjecture in the case where $P$ or $P^{\times}$ is a Cayley polytope of length $r$, where $r$ is the index of $P$ (see the definition below). For a related recent paper on stringy $E$-functions of Gorenstein polytopes in the case of $r=2$, see \cite{DoranNovoseltsev}.
\end{itemize}
\end{remark}

\begin{definition}\label{Cayley}
Let $M$ be a lattice and let $P_1,\ldots ,P_r$ be lattice polytopes in $M_{\mathbb{R}}$. We consider the lattice $\overline{M} := M\oplus \mathbb{Z}^r$ and we identify $P_i$ with $\overline{P_i} := P_i\times \{e_i\}$ in $\overline{M}_{\mathbb{R}}$, where $e_i$ is the $i$-th standard basis vector in $\mathbb{Z}^r$. The convex hull of $\overline{P_1},\ldots , \overline{P_r}$ is called the \textit{Cayley polytope} associated to $P_1,\ldots, P_r$. It is denoted by $P_1 * \cdots * P_r$. It is a lattice polytope with respect to the affine sublattice $M \times \{x_1 + \cdots + x_r = 1\}$ of $\overline{M}$. We will consider it as a lattice polytope with respect to $M\oplus \mathbb{Z}^{r-1}$ after deleting the last coordinate $x_r$. The cone 
\[  \mathbb{R}_{\geq 0} ( P_1 * \cdots * P_r) = \mathbb{R}_{\geq 0} \overline{P_1} + \cdots + \mathbb{R}_{\geq 0} \overline{P_r} \subset \overline{M}_{\mathbb{R}}    \]
is called the associated \textit{Cayley cone}. 
\end{definition}

Now we come to our main result. It will imply that the stringy $E$-function is a polynomial.

\begin{theorem}\label{maintheorem}
Let $P$ be a Gorenstein polytope of index $r$. Let $F$ be a face of $P$ and $F^{*}$ the face of $P^{\times}$ corresponding to $F$. Then
\[  r \leq \textnormal{codeg}\, F + \textnormal{codeg}\, F^{*}.      \]
\end{theorem}

This can be formulated equivalently as
\[ \deg F + \deg F^{*} \leq \deg P.\]
Since the degree of a lattice polytope is a measure of its `complexity', this means that $F$ and $F^{*}$ cannot both be `complicated'. The situation where equality holds for a face $F$ will be treated in Theorem~\ref{equality}.

\begin{proof}
Let $P$ be a polytope of dimension $d$. We may assume that it lives in $\widetilde{M}_{\mathbb{R}}$, with $\widetilde{M}$ a lattice of rank $d$. We put $s:=\text{codeg} \, F$ and $s' := \text{codeg}\, F^{*}$. It suffices to consider the case where $\emptyset \neq F \neq P$. We consider the lattice $M := \widetilde{M}\oplus \mathbb{Z}$ and the Gorenstein cone $C_P$ associated to $P$ in $M_{\mathbb{R}}$, i.e.\ the cone generated by the vertices of $\overline{P} := P\times \{1\}$.

Let $N$ be the lattice dual to $M$ and consider the dual cone $C_P^{\vee}$ in $N_{\mathbb{R}}$. The dual Gorenstein polytope $P^{\times}$ is the support polytope of $C_P^{\vee}$. Using the notations from Section~\ref{subsectionGorenstein}, we put $n:=n_{C_P} \in C_P^{\vee}\cap N$ and $m:=m_{C_P^{\vee}}\in C_P\cap M$. So $\langle n , m\rangle = r$. Let $f\in C_P \cap M$ be an interior lattice point of $sF$ considered as face of $s\overline{P}$. Similarly, let $f'\in C_P^{\vee} \cap N$ be an interior lattice point of $s'F^{*}$ considered as a face of $s'P^{\times}$. Then
\[ \langle n,f \rangle = s ,\quad \langle f',m \rangle = s' \quad  \text{and}\quad \langle f',f \rangle = 0.  \] 
It follows that
\[ r-s-s' = \langle n - f', m- f \rangle. \]
We will show that $\langle n - f', m- f \rangle \leq 0$.

Let $U$ be the vertex set of $P^{\times}$, $U_0$ the vertex set of $F^*$, and $U_1=U\setminus U_0$. In the same way we define $V$ as the vertex set of $P$, $V_0$ the vertex set of $F$, and $V_1:=V\setminus V_0$. Then there exist positive real numbers such that
\[ f' = \sum_{u\in U_0} \nu_u\ u, \quad n = \sum_{u\in U} \gamma_u\ u, \quad f = \sum_{v\in V_0} \mu_v\ v, \quad m = \sum_{v\in V} \lambda_v\ v.\]
We put $\nu_u = 0$ if $u\in U_1$ and $\mu_v = 0$ if $v\in V_1$. We define $\alpha_u := \gamma_u - \nu_u$ for all $u\in U$ and $\beta_v := \lambda_v - \mu_v$ for all $v\in V$, so that
\[ n -  f' = \sum_{u\in U} \alpha_u\ u,  \quad m-f = \sum_{v\in V} \beta_v\ v.\]

We claim that
\begin{eqnarray*}
u\in U_1 & \Longrightarrow & \sum_{v\in V} \beta_v\ \langle u,v\rangle \leq 0, \\
v\in V_1 & \Longrightarrow & \sum_{u\in U} \alpha_u\ \langle u,v\rangle \leq 0. \\
\end{eqnarray*}
By symmetry it is enough to show the first implication. So let $u\in U_1$. Since $u\notin F^*$ and since $f$ is an interior lattice point of $sF$, we get $\langle u,f \rangle >0$. Moreover, since $f$ is a lattice point, we get even more:
\[ \langle u, f\rangle \geq 1 = \langle u,m \rangle .\]
Hence, $\langle u, m-f \rangle \leq 0$, which implies the claim.

Now, let us put everything together: 
\[ \langle n - f', m- f \rangle = S_1 + S_2 + S_3 -S_4, \]
where 
\begin{itemize}
\item $S_1$ equals 
\[ \sum_{u\in U_1} \alpha_u \ \left( \sum_{v\in V} \beta_v\ \langle u,v \rangle \right) , \]
where $\alpha_u = \gamma_u > 0$ and the internal sum is $\leq 0$.
\item $S_2$ equals 
\[ \sum_{v\in V_1} \left( \sum_{u\in U} \alpha_u\ \langle u,v \rangle \right) \beta_v , \]
where $\beta_v = \lambda_v > 0$ and the internal sum is $\leq 0$.
\item $S_3$ equals 
\[ \sum_{u\in U_0} \sum_{v\in V_0} \alpha_u\ \beta_v\  \langle u,v\rangle, \]
which equals $0$, since here $\langle u,v \rangle =0$.
\item $S_4$ equals 
\[ \sum_{u\in U_1} \sum_{v\in V_1} \alpha_u\ \beta_v\  \langle u,v\rangle, \]
where $\alpha_u > 0, \beta_v > 0$, and $\langle u,v \rangle \geq 0$.
\end{itemize}
\vspace{-0,5cm}
\end{proof}

\begin{corollary}\label{proofofconjecture}
Let $P$ be a Gorenstein polytope of Calabi-Yau dimension $n$. Then $E_{st}(P;u,v)$ is a polynomial, and $E_{st}(P;u,v) = 0$ if $n< 0$.
\end{corollary}

\begin{proof}
Assume first that $n<0$. Then by definition we have that $\deg P < \text{codeg}\, P$. Let $F$ be a face of $P$. By Theorem~\ref{maintheorem}
\[ \deg F + \deg F^{*} \leq \deg P < \text{codeg}\, P \leq \text{codeg}\, F + \text{codeg}\, F^{*} \]
and hence $\deg F <  \text{codeg}\, F$ or $\deg F^{*} <  \text{codeg}\, F^{*}$. Corollary~\ref{corollaryStilde} implies that $\widetilde{S}(F,t)= 0$ or $\widetilde{S}(F^{*},t)=0$. It follows that $E_{st}(P;u,v)=0$.

Now let $n\geq 0$. Let $r$ be the index of $P$. We write $E_{st}(P;u,v)$ as $\sum_{\emptyset \leq F\leq P} A_1(F) \, A_2(F)$, where
\begin{eqnarray*}
A_1(F) & = & \frac{(-u)^{\dim F + 1}\,\widetilde{S}(F,u^{-1}v)}{(uv)^{r - \text{codeg}\, F^{*}}} \\
A_2(F)& = &  \frac{\widetilde{S}(F^{*},uv)}{(uv)^{\text{codeg}\, F^{*}}}. 
\end{eqnarray*}
Let $F$ be a face with both $A_1(F)$ and $A_2(F)$ nonzero. Then $A_2(F)$ is a polynomial by Corollary~\ref{corollaryStilde}. By Corollary~\ref{corollaryStilde} any monomial appearing in $A_1(F)$ is of the form $u^kv^l$ with 
\[  k = \dim F + 1 - m  - r + \text{codeg}\, F^* \]
and
\[ l = m  - r + \text{codeg}\, F^*\] 
for some $m$ with $\text{codeg}\, F \leq m \leq \deg F$.
Then
\[ k \geq  \dim F + 1 -\deg F - r + \text{codeg}\, F^* = \text{codeg}\, F - r + \text{codeg}\, F^* \geq 0\]
and
\[ l \geq  \text{codeg}\, F - r + \text{codeg}\, F^* \geq 0 \]
by Theorem~\ref{maintheorem}. Hence $A_1(F)$ is a polynomial as well.
\end{proof}

\section{Joins of Gorenstein polytopes}\label{joins}

In this section we study different kinds of joins of Gorenstein polytopes. This leads to a criterion for when there is a face of a Gorenstein polytope for which equality holds in Theorem~\ref{maintheorem}. We illustrate the definitions and results with many examples.

If $S$ is a subset of a vector space $V$ we denote by lin$(S)$ the linear subspace of $V$ spanned by $S$.

\begin{definition}
We say that a full dimensional lattice polytope $P$ in $M_{\mathbb{R}}$ is a \textit{join} of faces $F$ and $G$ if the following conditions hold:
\begin{itemize}
\item[(1)] $P$ is the convex hull of $F$ and $G$,
\item[(2)] $\dim F + \dim G = \dim P - 1$.
\end{itemize}
\end{definition}

\begin{remark} Equivalently, $M_{\mathbb{R}}\oplus \mathbb{R}$ is the direct sum of lin$(F\times \{1\})$ and lin$(G\times \{1\})$, and this splitting induces an isomorphism of $C_P$ and $C_F \oplus C_G$.
\end{remark}

\begin{definition}
Let $F$ and $G$ be faces of a full dimensional lattice polytope $P$ in $M_{\mathbb{R}}$. As in Definition~\ref{Cayley} we consider $F*G$ as a lattice polytope with respect to $M\oplus \mathbb{Z}$. We say that $P$ is a \textit{Cayley join} of $F$ and $G$ if the following conditions hold:
\begin{itemize}
\item[(1)] $\dim F +\dim G =  \dim P -1$,
\item[(2)] $F*G$ and $P$ are isomorphic, where $M$ is embedded in $M\oplus \mathbb{Z}$ as $M\times \{0\}$.
\end{itemize}
\end{definition}

\begin{remark}\label{separating}
If we use $x$ to denote the last coordinate in $M_{\mathbb{R}}\oplus \mathbb{R}$, then obviously the face $F*\emptyset$ of $F*G$ lies in the hyperplane $\{x=1\}$ and the face $\emptyset*G$ lies in $\{x = 0\}$. If $P$ is a Cayley join of $F$ and $G$ then it follows from (2) that there exists a lattice point $u$ in the dual lattice $N$ of $M$ and an integer $\delta$ such that 
\[ F \subset \{ v\in M \,|\, \langle u,v \rangle = \delta \} \quad \text{and} \quad G \subset \{ v\in M \,|\, \langle u,v \rangle = \delta - 1 \}.\]
\end{remark}

\begin{definition}
Let $F$ and $G$ be faces of a full dimensional lattice polytope $P$ in $M_{\mathbb{R}}$. We denote by $M(F)$ the sublattice $\text{lin}(F\times \{1\}) \cap (M \oplus \mathbb{Z})$ of $M\oplus \mathbb{Z}$, and similarly for $G$. We say that $P$ is a \textit{$\mathbb{Z}$-join} of $F$ and $G$ if the following conditions hold:
\begin{itemize}
\item[(1)] $P$ is the convex hull of $F$ and $G$,
\item[(2)] the natural map $M(F) \oplus M(G) \to M \oplus \mathbb{Z}$ is an isomorphism.
\end{itemize}
\end{definition}

\begin{remark}\label{RemarkJoins} Several remarks are in order.
\begin{itemize}
\item[(1)] It is clear that a Cayley join is a join. A $\mathbb{Z}$-join is a Cayley join. This can be seen as follows. We embed $P$ as $P\times \{1\}$ in $(M\oplus \mathbb{Z})_{\mathbb{R}}$ and we choose a basis of this lattice by taking the union of bases $\{e_1,\ldots ,e_{\dim F +1}\} $ and $\{f_1,\ldots , f_{\dim G +1}\}$ for $M(F)$ respectively $M(G)$. Moreover, we assume that $e_1,\ldots ,e_{\dim F +1}$ lie in the affine span of $F\times \{1\}$. Then we consider the dual basis $\{ e_1^*,\ldots ,f_{\dim G +1}^* \}$ of the dual lattice. Then $F\times \{1\}$ and $G\times\{1\}$ lie on two parallel hyperplanes determined by $e_1^* + \cdots + e_{\dim F +1}^*$. It follows that $P$ is a Cayley join of $F$ and $G$.
\item[(2)] Let $P$ be a $\mathbb{Z}$-join of faces $F$ and $G$. We show that $P$ is isomorphic to what is called in the literature the join, or free join, of $F$ and $G$ (see \cite[p.362]{HenkRichterZiegler}). The \textit{free join} $F\star G$ of $F$ and $G$ is defined as the convex hull of 
\[ (F\times \{1\}) \times \{0\} \times \{0\} \qquad \text{and} \qquad \{0\} \times (G\times \{1\}) \times\{1\} \] in $(M(F)\oplus M(G) \oplus \mathbb{Z})_{\mathbb{R}}$. Usually one works with full dimensional polytopes, and then the last coordinate is used to make sure that they are embedded in skew affine spaces. In our case, we can forget the last coordinate because $F\times \{1\}$ and $G\times\{1\}$ already lie in skew affine spaces, since $\dim F + \dim G = \dim P - 1$.
\item[(3)] For a free join, and hence for a $\mathbb{Z}$-join, it is well known that the $h^*$-polynomial behaves multiplicatively, i.e.\ $h_{P}^*(t) = h_F^*(t) \,h_G^*(t)$, see for instance \cite[Lemma 1.3]{HenkTagami}. Using Theorem~\ref{resultHibi} it follows immediately that a $\mathbb{Z}$-join of Gorenstein polytopes is Gorenstein.
\item[(4)] For all three kinds of joins that we have defined, one has that every face of $P$ is a join of the same kind of a face of $F$ and a face of $G$.
\item[(5)] For a free join or a $\mathbb{Z}$-join we also have $\widetilde{S}(P,t) = \widetilde{S}(F,t) \, \widetilde{S}(G,t)$. This follows immediately from (3), (4) and the multiplicativity of $g$-polynomials with respect to products of posets.
\item[(6)] Let $P$ be a join of $F$ and $G$. As above, we can embed the free join of $F$ and $G$ in $M(F)\oplus M(G)$, which embeds as a sublattice of finite index in $M\oplus \mathbb{Z}$ since $M(F)_{\mathbb{R}}\oplus M(G)_{\mathbb{R}} \cong M_{\mathbb{R}} \oplus \mathbb{R}$. It follows that we can consider a join as a free join with respect to a coarser lattice, and hence by (3) we have
\[ \text{codeg}\, P \leq \text{codeg}\, F\star G = \text{codeg}\, F + \text{codeg}\, G.\] 
\end{itemize}
\end{remark}

\begin{example}\label{ExampleJoins}
We illustrate the above definitions with several examples.
\begin{itemize}
\item[(1)] The interval $[0,2]$ is a lattice polytope with respect to $\mathbb{Z}$. It is a join of its vertices, but no Cayley join.
\item[(2)] Consider the lattice $\mathbb{Z}^2$ in $\mathbb{R}^2$ and the lattice polytopes
\begin{eqnarray*}
F &\text{ with vertices }& (0,0), (1,0),\\
G &\text{ with vertices }& (0,0), (-1,2).
\end{eqnarray*}
The Cayley polytope $P=F*G$ is a lattice polytope with respect to $\mathbb{Z}^3$ and has vertex set
\[ \{ (0, 0, 1), (1, 0, 1), (0, 0, 0), (-1, 2, 0) \} \]
and $P$ is a Cayley join of its two $1$-dimensional faces $F$ and $G$. 
The reader may check that $M(F)\oplus M(G)$ embeds as a sublattice of index 2 in $\mathbb{Z}^4$ and hence $P$ is not a $\mathbb{Z}$-join of $F$ and $G$.
Note that $h^*_F(t) = h_G^*(t) = 1$ and that $h_P^*(t)=1+t^2$, so all three polytopes are Gorenstein, but the $h^*$-polynomial does not behave multiplicatively.
\item[(3)] A Cayley join of Gorenstein polytopes does not need to be Gorenstein. If in (2) we change the last vertex of $G$ to $(-1,3)$, then $h_{F*G}^*(t)$ becomes $1+2t^2$ and hence $F*G$ is not Gorenstein.
\end{itemize}
\end{example}

\begin{lemma}\label{dual-join}
Let $P$ be a Gorenstein polytope that is a join of $F$ and $G$. 
Then the dual Gorenstein polytope $P^{\times}$ is a join of $F^*$ and $G^*$. 
Here, the polytopes $F$ and $G^*$, respectively $G$ and $F^*$, are combinatorially dual (i.e.\ their face posets are dual posets).
\end{lemma}

\begin{proof}
Since $\dim F + \dim F^* = \dim P-1$, and similarly for $G$ and $G^*$, it follows immediately that $\dim F^* + \dim G^* = \dim P -1$. Any facet of $P$ contains by Remark~\ref{RemarkJoins} (4) either $F$ or $G$. This means that any vertex of $P^{\times}$ belongs to $F^*$ or $G^*$, and hence $P^{\times}$ is the convex hull of $F^*$ and $G^*$. 

The duality between faces of $F$ and $G^*$ works as follows. Starting from a face $H$ of $F$ we consider the face $H'$ of $P$ that is a join of $H$ and $G$. Then $(H')^*$ is a face of $G^*$. 
\end{proof}

Let $P$ be a (full dimensional) Gorenstein polytope in $M_{\mathbb{R}}$ that is a join of $F$ and $G$. We write $\overline{M}$ for $M\oplus \mathbb{Z}$ and we consider the cone $C_P$ in $\overline{M}_{\mathbb{R}}$. Let $N$ be the lattice dual to $M$, and identify the dual of $\mathbb{Z}$ with $\mathbb{Z}$ via $\alpha \mapsto \alpha(1)$. We write also $\overline{N}$ for $N\oplus \mathbb{Z}$. The face $G^*$ of $P^{\times}$ (as support polytope of $C_P^{\vee}$ in $\overline{N}_{\mathbb{R}}$) determines the sublattice $N(G^*) := \text{lin}(G^*) \cap \overline{N}$ of $\overline{N}$. We get the (split) exact sequence of free $\mathbb{Z}$-modules
\[  0 \to N(G^*) \to \overline{N} \to \overline{N}/N(G^*) \to 0.\] 
Note that the dual lattice of $\overline{N}/N(G^*)$ is the sublattice $M(G)$ of $\overline{M}$ (with the obvious induced pairing). So we get the dual exact sequence
\[  0 \to M(G) \to \overline{M} \to \overline{M}/M(G) \to 0.\] 
Since $\overline{M}_{\mathbb{R}} = M(F)_{\mathbb{R}} \oplus M(G)_{\mathbb{R}}$, we find that $M(F)$ is a sublattice in $\overline{M}/M(G)$ of finite index. We denote the inclusion and its real extension by $\pi_G$ and we will consider $\pi_G(F)$ henceforth as a lattice polytope with respect to $\overline{M}/M(G)$. Of course, we can repeat this story to construct the polytope $\pi_F(G)$ in $\overline{M}/M(F)$.

After these preparations we can formulate the following proposition.

\begin{proposition}\label{cayley-join}
Let $P$ be a Gorenstein polytope in $M_{\mathbb{R}}$ that is a Cayley join of faces $F$ and $G$. 
Then $G^*$ is a Gorenstein polytope with respect to $N(G^*)$ 
with dual Gorenstein polytope $\pi_G(F)$. 
Similarly, $F^*$ is a Gorenstein polytope 
with respect to $N(F^*)$ with dual Gorenstein polytope $\pi_F(G)$.

If $P$ is a $\mathbb{Z}$-join of $F$ and $G$, then $G^*$ and $F$, respectively $F^*$ and $G$ are 
dual Gorenstein polytopes.
\end{proposition}

\begin{proof}
We consider the Gorenstein cone $C_{G^*}$ generated by the vertices of $G^*$ in $N(G^*)_{\mathbb{R}}$. It is clear that $C_{G^*}^{\vee}$ in $\bigl(\overline{M}/M(G)\bigr)_{\mathbb{R}}$ equals the cone $\pi_G(C_F)$. 

As in the proof of Theorem~\ref{maintheorem} we denote the vertex set of $F$ by $V_0$ and the vertex set of $G$ by $V_1$. By Remark~\ref{separating} there exists a lattice point $y'\in \overline{N}$ and an integer $\delta$ such that $\langle y',V_0 \rangle = \delta$ and $\langle y',V_1 \rangle = \delta - 1$. We put $y = y' - (\delta-1) n_{C_P}$, with $n_{C_P}$ as in Section~\ref{subsectionGorenstein}. Then $\langle y,V_0 \rangle = 1$ and $\langle y,V_1 \rangle = 0$. Hence $y$ lies in the interior of $ C_{G^*}$ in $N(G^*)_{\mathbb{R}}$. In particular, $\langle y, \pi_G(v_0) \rangle = 1$ for all $v_0\in V_0$. This means that the cone $C_{G^*}^{\vee}$ in $\bigl(\overline{M}/M(G)\bigr)_{\mathbb{R}}$ is a Gorenstein cone with support polytope $\pi_G(F)$. Hence, $G^*$ and $\pi_G(F)$ are dual Gorenstein polytopes. 

In case $P$ is a $\mathbb{Z}$-join of $F$ and $G$, then $\pi_G$ is an isomorphism at the level of the lattices and hence $G^*$ and $F$ are dual Gorenstein polytopes.
\end{proof}

\begin{remark}
In particular, we observe that in Remark~\ref{RemarkJoins} (3) also the converse holds: a $\mathbb{Z}$-join of polytopes $F,G$ is Gorenstein if and only if $F,G$ are Gorenstein.
\end{remark}

\begin{example}\label{ExampleNonGorenstein}
From Example~\ref{ExampleJoins} (2) we cook up an example of a Gorenstein polytope that is a Cayley join of two faces such that not both faces are Gorenstein. Consider the lattice $\mathbb{Z}^3$ in $\mathbb{R}^3$ and the lattice polytopes
\begin{eqnarray*}
F &\text{ of dimension $1$ with vertices }& (0,0,0), (1,0,0),\\
G &\text{ of dimension $2$ with vertices }& (0,0,0), (-1,2,0), (0,0,2), (-1,2,2).
\end{eqnarray*}
The $4$-dimensional Cayley polytope $P=F*G$ is Gorenstein of index 2 since $h^*_P(t) = 1 +3t + 3t^2+t^3$, whereas $G$ is not Gorenstein.
\end{example}

Now we can treat the case where equality holds in Theorem~\ref{maintheorem} for a certain face.

\begin{theorem}\label{equality}
Let $F$ be a face of a Gorenstein polytope $P$ of index $r$. The following are equivalent:
\begin{itemize}
\item[(1)] $\textnormal{codeg}\, F + \textnormal{codeg}\, F^* = r$,
\item[(2)] there exists a face $G$ of $P$ such that $P$ is a Cayley join of $F$ and $G$, and $P^{\times}$ is a Cayley join of $F^*$ and $G^*$.
\end{itemize}
In this case, $F,G,F^*,G^*$ are Gorenstein polytopes, and 
\[ \textnormal{codeg}\, P = \textnormal{codeg}\, F + \textnormal{codeg}\, G. \]
The same additivity property holds for the degree and the Calabi-Yau dimension. 
\end{theorem}

\begin{proof}
We use the notations of the proof of Theorem~\ref{maintheorem}. Note that $\alpha_u > 0$ for $u\in U_1$ and $\beta_v > 0$ for $v\in V_1$.
It follows that (1) holds if and only if there exist lattice points $f$ in the relative interior of $sF \subset C_P$ and $f'$ in the relative interior of $s'F^* \subset C_P^{\vee}$ such that
\begin{itemize}
\item[(i)] $\langle u,v \rangle = 0$ for all $u\in U_1, v\in V_1$,
\item[(ii)] $\langle u, m - f \rangle = 0$ for all $u\in U_1$, or equivalently $\langle u, f \rangle = 1$ for all $u\in U_1$,
\item[(iii)] $\langle n-f',v \rangle = 0$ for all $v\in V_1$, or equivalently $\langle f',v \rangle = 1$ for all $v\in V_1$.
\end{itemize}

Assume first that these three conditions hold. Then $n-f'$ cuts out a face of $C_P$, corresponding to a face $G$ of $P$, with $G$ equal to the convex hull of all vertices in $V_1$. Moreover, $G^*$ has $U_1$ as vertex set. Since
\[  F = \{ x\in P \, | \, \langle f',x \rangle  = 0 \} \qquad \text{and} \qquad G = \{ x\in P \, | \, \langle f',x \rangle  = 1\},  \]
it follows that $P$ is isomorphic to the Cayley polytope associated to $F$ and $G$. Similarly, $P^{\times}$ is isomorphic to the Cayley polytope associated to $F^*$ and $G^*$. Let $d:=\dim P$. Since $P$ is the convex hull of the vertices of $F$ and $G$ we have $\dim F + \dim G \geq d -1$ and similarly $d-1 \leq \dim F^* + \dim G^*$. So, $d- 1 \leq d-1 - \dim F + d-1 - \dim G \leq d-1$. Therefore, $\dim F + \dim G = d -1$ as well as $\dim F^* + \dim G^* = d -1$. 

Assume now that (2) holds. We show that (i),(ii) and (iii) hold true then. We have statement (i) by assumption, since $U_1$ is the vertex set of $G^*$ and $V_1$ the vertex set of $G$. Note that all polytopes $F,G,F^*,G^*$ are Gorenstein by Proposition~\ref{cayley-join}.

By Remark~\ref{separating} there exist lattice points $x'\in N\oplus \mathbb{Z}$ such that $\langle x', V_0\rangle  = 0$ and $\langle x', V_1\rangle  = 1$, and $y' \in N\oplus \mathbb{Z}$ such that 
$\langle y', V_0\rangle  = 1$ and $\langle y', V_1\rangle  = 0$. Hence, $x'$ belongs to the interior of $C_{F^*}$ and 
$y'$ to the interior of $C_{G^*}$ such that $x'+y' = n$. Let $w$ be any lattice point in the interior of $C_F$. 
Then $\langle u_1,w\rangle \geq 1$ for any $u_1 \in U_1$. Therefore, $\langle y'/\langle y',m\rangle , w\rangle \geq 1$.
So $\langle y',w\rangle \geq \langle y',m \rangle$. Therefore, $\langle n,w \rangle = \langle y',w\rangle \geq \langle y',m\rangle$. 
In particular, $\text{codeg}\, F = \langle n,f \rangle \geq \langle y', m \rangle$, where $f$ is the unique lattice point in the relative interior of $sF$. Now, let $x \in M \oplus \mathbb{Z}$ such that $\langle U_0, x \rangle  = 0$ and $\langle U_1,x \rangle = 1$. Therefore, 
$x$ is a lattice point in the interior of $C_F$. In particular, we get $\text{codeg}\, F \leq \langle n,x \rangle$. 
Moreover, the previous argument yields $\langle  n,x \rangle  = \langle y',x \rangle = \langle y', m \rangle \leq \text{codeg} \, F$. 
Therefore, $\langle n,x\rangle = \text{codeg} \, F$. This shows that $x = f$ and hence (ii) holds. Of course, the proof of (iii) is similar.

We have shown that $\langle y', m \rangle = \text{codeg}\, F$. A symmetric argument shows that $\langle x', m \rangle = \text{codeg}\, G$. From $x'+y' = n$ we get then 
\[ \text{codeg}\, F + \text{codeg}\, G  = \langle n,m \rangle = r = \text{codeg}\, P .\] 
From the definitions it follows that this additivity holds for the degree and the Calabi-Yau dimension as well.
\end{proof}



This theorem inspires us to give the following definition. We investigate the relation with irreducible nef-partitions in Section~\ref{SectionNef}.

\begin{definition}\label{Irreducible}
Let $P$ be a nonempty Gorenstein polytope of index $r$. We call $P$ \textit{irreducible} if there is no proper face $F$ with $\text{codeg}\, F + \text{codeg}\, F^* = r $. 
\end{definition}

In particular, any reflexive polytope is irreducible, since it is Gorenstein of index $1$. In the situation of the above theorem, $F$ and $G^*$ do not need to be dual Gorenstein polytopes, as the following example shows.

\begin{example}\label{BatyBoriExample}
This example is based on \cite[Example 5.5]{BatyrevBorisovCompleteIntersections}. Consider the lattice $(\frac{1}{2},\frac{1}{2},\frac{1}{2},\frac{1}{2}) +  \mathbb{Z}^4$ in $\mathbb{R}^4$ and the isomorphic $2$-dimensional lattice polytopes
\begin{eqnarray*}
F &\text{ with vertices }& (1,0,0,0), (0,1,0,0), (-1,0,0,0), (0,-1,0,0),\\
G &\text{ with vertices }& (0,0,1,0), (0,0,0,1), (0,0,-1,0), (0,0,0,-1).
\end{eqnarray*}
One can check that $P=F*G$ is $5$-dimensional Gorenstein of index 2, with 
\[h_P^*(t) = 1+4t+22t^2+4t^3+t^4.\]
It is easy to see that both $F$ and $G$ are Gorenstein of index 1, with $h_F^*(t) = h_G^*(t) = 1 + 2t + t^2$. Note that $h_P^*(t)$ differs from $(1 + 2t + t^2)^2 = 1 + 4t + 6t^2 + 4t^3 +t^4$ as one would get for a $\mathbb{Z}$-join by considering the standard lattice $\mathbb{Z}^4$.

One can also check that $P$ is self-dual, i.e.\ $P$ and $P^{\times}$ are isomorphic lattice polytopes. Moreover, the face $G^*$ of $P^{\times}$ is isomorphic to $F$, whereas it is easy to see that $F$ is not isomorphic to its dual Gorenstein polytope. 

A somewhat longer computation shows that 
\[ \widetilde{S}(P,t) = t^4 + 2t^3 +t^2,\]
and $\widetilde{S}(F,t) = \widetilde{S}(G,t) = t^2+t$. The faces that contribute to the stringy $E$-function of $P$ are $\emptyset, F,G,P$ and hence
\[ E_{st}(P;u,v)  =  (uv)^2 - 2u^2v - 2uv^2 +4uv +u^2+v^2 -2u -2v +1 ,\]
and $E_{st}(F;u,v)=  E_{st}(G;u,v) = uv -u-v+1$, so 
\[ E_{st}(P;u,v) = E_{st}(F;u,v)\, E_{st}(G;u,v)\]
in this example. 
\end{example}

We note that the stringy $E$-function behaves multiplicatively with respect to $\mathbb{Z}$-joins.

\begin{proposition}\label{stringyjoin}
If a Gorenstein polytope $P$ is a $\mathbb{Z}$-join of faces $F$ and $G$, then \[E_{st}(P;u,v) = E_{st}(F;u,v) \, E_{st}(G;u,v).\]
\end{proposition}

\begin{proof}
We have that $P^{\times}$ is also a $\mathbb{Z}$-join of $F^*$ and $G^*$. For faces $F'$ of $F$ and $G'$ of $G$ we have a face $P'$ of $P$ that is a $\mathbb{Z}$-join of $F'$ and $G'$. Moreover, $(P')^*$ is a $\mathbb{Z}$-join of the face $(F')^*$ of the dual Gorenstein polytope $F^{\times} = G^*$ corresponding to $F'$ and the face $(G')^*$ of the dual Gorenstein polytope $G^{\times} = F^*$ corresponding to $G$. It suffices now to use the multiplicativity of the $\widetilde{S}$-polynomial with respect to $\mathbb{Z}$-joins. 
\end{proof}

\begin{example}\label{HardExample}
Surprisingly, in the situation of Theorem~\ref{equality} it can still happen that the stringy $E$-function does not behave multiplicatively. 
We start from Example~\ref{ExampleJoins} (2) and then we mimic the example of Batyrev and Borisov somewhat. Consider the lattice $M = (0,0,\frac{1}{2},0,0,\frac{1}{2}) + \mathbb{Z}^6$ in $\mathbb{R}^6$ and the isomorphic $3$-dimensional lattice polytopes 
\begin{eqnarray*}
&F\quad \text{with vertices }& a_1 = (0,0,1,0,0,0), a_2 = (1,0,1,0,0,0),\\
& &  a_3 = (0,0,0,0,0,0), a_4 = (-1,2,0,0,0,0), \\
&G\quad \text{with vertices }& b_1 = (0,0,0,0,0,1), b_2 = (0,0,0,1,0,1),\\
& & b_3 = (0,0,0,0,0,0), b_4 = (0,0,0,-1,2,0). 
\end{eqnarray*}
Let $P$ be the $7$-dimensional Cayley polytope associated to $F$ and $G$, lying in $\mathbb{R}^7$, considered with respect to $M\oplus \mathbb{Z}$.
One may check that
\[ h_F^*(t) = h_G^*(t) = 1+t^2,\] 
\[ \widetilde{S}(F,t) = \widetilde{S}(G,t) = t^2,\] 
\[ E_{st}(F;u,v) = E_{st}(G; u,v ) = 2.\]
One may check as well that 
\[ h_P^*(t) = 1 + 6t^2 + t^4,\quad h_{P^{\times}}^* = 1 + 4t + 22 t^2 + 4t^3 + t^4,\]
\[  \widetilde{S}(P,t) = \widetilde{S}(P^{\times},t) = t^4.\]
So $P$ is Gorenstein of index $4$ and Calabi-Yau dimension $0$.

We have $h_{F^*}^* (t) = 1+ 2t +t^2$, and hence $\text{codeg} \, F + \text{codeg}\, F^* = \text{codeg}\, P$. So $P^{\times}$ is also a Cayley join of $F^*$ and $G^*$. Note that the following sets of vertices determine six 3-dimensional faces of $P$ which all have $h^*$-polynomial equal to $1 + t^2$, and $\widetilde{S}$-polynomial equal to $t^2$:
\begin{eqnarray*}
& & \{a_1,a_2,a_3,a_4\} , \{ b_1,b_2,b_3,b_4\} ,  \{a_1,a_3,b_1,b_3\}, \\
& & \{a_1,a_3,b_2,b_4 \},\{a_2,a_4,b_1,b_3 \},\{a_2,a_4,b_2,b_4 \}.
\end{eqnarray*}
Moreover, their dual faces all have $h^*$-polynomial $1+2t  + t^2$ and $\widetilde{S}$-polynomial equal to $t^2$. These faces, together with $\emptyset$ and $P$, contribute to the stringy $E$-polynomial to give $E_{st}(P;u,v) = 8$.
\end{example}

\section{Gorenstein polytopes and nef-partitions}\label{SectionNef}

In this section we recall the relation between nef-partitions of a reflexive polytope and Gorenstein polytopes. We show that irreducible nef-partitions correspond to irreducible Gorenstein polytopes in Theorem~\ref{IrreducibleProposition}.

\begin{definition}\label{nefpartition}
Let $P$ be a reflexive polytope in $M_{\mathbb{R}}$. A Minkowski sum decomposition $P = P_1 + \cdots + P_r$ of $P$ in nonempty lattice polytopes $P_1,\ldots, P_r$ is called a \textit{nef-partition} if $0\in P_i$ for each $i$. Since $P$ is determined by $P_1,\ldots , P_r$, we will denote a nef-partition simply by the set $\Pi(P):=\{ P_1, \ldots , P_r\}$.
\end{definition}

We note that this is called a \textit{centered} nef-partition in \cite[Def.\ 3.1]{BatyrevNill}. Next we recall the definition of a special simplex. 

\begin{definition}
Let $P$ be a lattice polytope in $M_{\mathbb{R}}$. A simplex $S$ spanned by $r$ affinely independent lattice points in $P$ is called a \textit{special $(r-1)$-simplex} of $P$ if every facet of $P$ contains precisely $r-1$ vertices of $S$.
\end{definition}

Because the codegree of a simplex of dimension $k$ is at most $k+1$, we immediately have the following.

\begin{lemma}\label{LemmaSpecialSimplex}
If a lattice polytope $P$ has a special $k$-simplex, then $\textnormal{codeg}\, P \leq k+1$.
\end{lemma}

\begin{proposition}\label{NefPartitionGivesGorensteinCayley}
\textnormal{\cite[Prop.\ 3.6 and Cor.\ 3.7]{BatyrevNill}} Let $M$ be a lattice of rank $d$. Let $P$ be a reflexive polytope in $M_{\mathbb{R}}$ and let $\Pi(P) = \{P_1, \ldots , P_r\}$ be a nef-partition. Then the associated Cayley polytope $\widetilde{P} := P_1 * \cdots * P_r$ is a Gorenstein polytope of dimension $d+r-1$ and index $r$, and both $\widetilde{P}$ and $(\widetilde{P})^{\times}$ have a special $(r-1)$-simplex. 
\end{proposition}

\begin{remark}\label{RemarkSpecialSimplices}
The special $(r-1)$-simplices are constructed as follows. We consider $\widetilde{P}$ as the support polytope of the Cayley cone $C$ associated to $P_1,\ldots,P_r$ in $(M\oplus \mathbb{Z}^r)_{\mathbb{R}}$ and we denote the standard basis vectors of $\mathbb{Z}^r$ with $e_1,\ldots , e_r$. Then the points $0\times e_i$ form a special $(r-1)$-simplex in $\widetilde{P}$. After choice of the dual basis $e_1^*,\ldots , e_r^*$ to $e_1,\ldots , e_r$ we identify the dual of $\mathbb{Z}^r$ with $\mathbb{Z}^r$. Then $e_1^*,\ldots , e_r^*$ form the vertices of a special $(r-1)$-simplex in $(\widetilde{P})^{\times}$, considered as support polytope of $C^{\vee}$ in $N\oplus \mathbb{Z}^r$.

Note that we recover the polytopes $P_i\times \{e_i\}$ simply by
\[ P_i\times \{e_i\} = \{ p\in P\,|\, \langle e_i^* , p \rangle = 1\} = \{p\in P\,|\, \langle e_j^* , p \rangle = 0\text{ for } j\neq i\}.\]
\end{remark}

The converse of Proposition~\ref{NefPartitionGivesGorensteinCayley} can be formulated as follows.

\begin{proposition}\label{Converse}
\textnormal{\cite[Thm.\ 2.6 and Prop.\ 3.6]{BatyrevNill}} Let $M$ be a lattice of rank $d$ and let $P_1,\ldots,P_r$ be lattice polytopes in $M_{\mathbb{R}}$ such that $P:= P_1+\cdots + P_r$ has dimension $d$. Assume that the associated Cayley polytope $\widetilde{P}:= P_1* \cdots *P_r$ is Gorenstein of index $r$. Then $P$ has a unique interior lattice point $m$ and $P-m$ is reflexive. If in addition $\widetilde{P}$ has a special $(r-1)$-simplex with vertices $p_i\times e_i$ for $i= 1,\ldots, r$, then $p_1+\cdots +p_r = m$ and $\{P_1-p_1,\ldots,P_r-p_r\}$ is a nef-partition of $P-m$.
\end{proposition}

\begin{definition} 
\cite[Def.\ 5.6]{BatyrevBorisovCompleteIntersections} Let $P$ be a reflexive polytope and let $\Pi(P)= \{P_1,\ldots , P_r\}$ be a nef-partition of $P$. Then $\Pi(P)$ is called \textit{irreducible} if there is no subset $\{i_1 , \ldots , i_s \} \subset \{ 1 , \ldots ,r \}$ such that $P_{i_1} + \cdots + P_{i_s}$ contains 0 in its relative interior.
\end{definition}

\begin{theorem}\label{IrreducibleProposition}
Let $\Pi(P)= \{ P_1, \ldots , P_r\}$ be a nef-partition of a reflexive polytope $P$ in $M_{\mathbb{R}}$. Then $\Pi(P)$ is irreducible if and only if the associated Cayley polytope $\widetilde{P} = P_1 * \cdots * P_r$ is an irreducible Gorenstein polytope.
\end{theorem}

\begin{proof}
Assume first that $\Pi(P)$ is reducible. Then there exists a subset 
\[\{i_1,\ldots , i_s\} \subset \{1,\ldots , r\}\]
such that $P_1:= P_{i_1} + \cdots + P_{i_s}$ contains $0$ in its relative interior. Let $\{j_1,\ldots , j_t\}$ be the complement of $\{i_1,\ldots , i_s\}$ in $\{1,\ldots , r\}$. From \cite[Prop.\ 6.11 and 6.13]{BatyrevNill} it follows that $P_1$ and $P_2:=P_{j_1} + \cdots + P_{j_t}$ are both reflexive and that $\dim P_1 + \dim P_2 = \dim P$. From Proposition~\ref{NefPartitionGivesGorensteinCayley} we conclude that the associated Cayley polytopes $\widetilde{P_1}$ and $\widetilde{P_2}$ are Gorenstein polytopes of indices $s$, respectively $t$.


The Cayley polytope $\widetilde{P}$ clearly has a face $F_1$ that is isomorphic to $\widetilde{P_1}$ and a face $F_2$ that is isomorphic to $\widetilde{P_2}$, such that $\widetilde{P}$ is a Cayley join of $F_1$ and $F_2$. We also know that $\widetilde{P}$ is a Gorenstein polytope of index $r=s+t$. From Lemma~\ref{dual-join} it follows that $(\widetilde{P})^{\times}$ is a join of $F_1^*$ and $F_2^*$. It suffices to show that $\text{codeg}\, F_1^* = t$. Proposition~\ref{cayley-join} implies that $F_1^*$ is Gorenstein with dual $\pi_{F_1}(F_2)$. It is trivial that $\text{codeg}\, \pi_{F_1}(F_2) \leq \text{codeg}\, F_2$ and hence $\text{codeg}\, F_1^* \leq t$ since dual Gorenstein polytopes have the same codegree. Similarly $\text{codeg}\, F_2^* \leq s$. On the other hand, we have by Remark~\ref{RemarkJoins} (6) that
\[ r = \text{codeg}\, (\widetilde{P})^{\times} \leq \text{codeg}\, F_1^* + \text{codeg}\, F_2^* \leq t + s = r,\]
and hence equality holds everywhere. We conclude that $\widetilde{P}$ is reducible.

Conversely, assume that $\widetilde{P}$ is reducible. So there are nonempty faces $F$ and $G$ of $\widetilde{P}$ such that $\widetilde{P}$ is a Cayley join of $F$ and $G$ and such that $(\widetilde{P})^{\times}$ is a Cayley join of $F^*$ and $G^*$. In addition, all of these faces are Gorenstein. We consider $\widetilde{P}$ as support polytope of its Cayley cone and we use the notations of Remark~\ref{RemarkSpecialSimplices}. Since $(\widetilde{P})^{\times}$ is a Cayley join of $F^*$ and $G^*$, all vertices $e_1^*, \ldots ,e_r^*$ of the special simplex of $(\widetilde{P})^{\times}$ are contained in either $F^*$ or $G^*$. Without loss of generality we may assume that $G^*$ contains $e_1^*,\ldots,e_k^*$ with $0<k<r$ and that $F^*$ contains $e_{k+1}^*,\ldots , e_r^*$. It follows immediately that these vertices determine a special $(k-1)$-simplex of $G^*$ and a special $(r-k-1)$-simplex of $F^*$. By Theorem~\ref{equality} and Lemma~\ref{LemmaSpecialSimplex} we also have
\[ r = \text{codeg}\, (\widetilde{P})^{\times} = \text{codeg}\, F^* + \text{codeg}\, G^* \leq r- k + k = r, \]
and hence we have equality everywhere.

Since $F^*$ contains $e_{k+1}^*,\ldots,e_r^*$, we see that $F$ is contained in the face of $\widetilde{P}$ that is isomorphic to the Cayley polytope $P_1*\cdots * P_k$. Similarly $G$ is contained in the face isomorphic to $P_{k+1}*\cdots * P_r$. But these faces are disjoint and $\widetilde{P}$ is a Cayley join of $F$ and $G$, so it follows that these containments are equalities. The index of $F$ as a Gorenstein polytope equals $r - \text{codeg}\, F^* = r-(r-k)= k$. Proposition~\ref{Converse} tells us now that $P_1 + \cdots + P_k$ has a unique lattice point in its relative interior, which is necessarily $0$. Hence $\Pi(P)$ is reducible.
\end{proof}

Batyrev and Borisov have shown in \cite[Thm.\ 5.8]{BatyrevBorisovCompleteIntersections} that a nef-partition is reducible if and only if it splits as a direct sum, in the following sense.

\begin{definition}
\cite[Def.\ 5.1 and 5.2]{BatyrevBorisovCompleteIntersections} Let $P$ be a reflexive polytope in $M_{\mathbb{R}}$ and let $\Pi(P)=\{P_1,\ldots ,P_r\}$ be a nef-partition. We say that $\Pi(P)$ \textit{splits as a direct sum}
if there is a partition $I_1,\ldots,I_k$ of the index set $I=\{1,\ldots, r\}$ such that
for all $j$, $\{P_i\,|\, i\in I_j\}$ forms a nef-partition of a reflexive polytope $P_{I_j}$.
Here $P_{I_j}$ is considered as a lattice polytope with respect to $M_j := \text{lin}(P_{I_j}) \cap M$. We write
\[ \Pi(P) = \Pi(P_{I_1}) \oplus \cdots \oplus \Pi(P_{I_k}) . \]

By \cite[Prop.\ 6.11]{BatyrevNill} we automatically have that $\dim P_{I_1} + \cdots + \dim P_{I_k} = \dim P$. 
We say that $\Pi(P)$ \textit{splits as a direct sum over} $\mathbb{Z}$ if in addition 
\[ M = M_1 \oplus \cdots \oplus M_k .\]
\end{definition}

We obtain the following.

\begin{proposition}\label{ZSplit}
Let $P$ be a reflexive polytope in $M_{\mathbb{R}}$ and let $\Pi(P)= \{ P_1 , \ldots , P_r\}$ be a nef-partition. Then $\Pi(P)$ splits as a direct sum over $\mathbb{Z}$,
\[ \Pi(P) =   \Pi(P_{I_1}) \oplus \cdots \oplus \Pi(P_{I_k}),      \]
if and only if the associated Cayley polytope $\widetilde{P} = P_1* \cdots * P_r$ is a $\mathbb{Z}$-join of its faces corresponding to $\widetilde{P_{I_1}} ,\ldots , \widetilde{P_{I_k}}$. 
\end{proposition}

\begin{remark}
It was shown in \cite[Prop.\ 6.15]{BatyrevNill} that a nef-partition has a unique decomposition as a direct sum of irreducible nef-partitions. This means that the Cayley polytope associated to a nef-partition can be uniquely written as a Cayley join of irreducible Gorenstein polytopes. For general Gorenstein polytopes, this uniqueness does not hold. This follows for instance from Example~\ref{HardExample}: the six $3$-dimensional faces listed there are all irreducible Gorenstein polytopes and they form three pairs such that $P$ is a Cayley join of each of the pairs.
\end{remark}

\begin{remark}
Let $\Pi(P) = \{P_1,\ldots , P_r\}$ be a nef-partition which splits as a direct sum over $\mathbb{Z}$ as
\[ \Pi(P) = \{ P_1 \} \oplus \cdots \oplus \{P_r\}  . \]
Let $Q$ be the convex hull of $P_1,\ldots , P_r$. Here $Q$ is also called a \textit{free sum}. Then $h^*_Q(t) = h^*_{P_1}(t) \cdot \ldots \cdot h^*_{P_r}(t)$. This is a special case of the formula in \cite{Braun}. Here it can be seen directly as follows. By Proposition~\ref{ZSplit} and Remark~\ref{RemarkJoins} (3) the $h^*$-polynomial of $P_1 * \cdots * P_r$ is the product of the $h^*$-polynomials of $P_1,\ldots , P_r$. Now $Q$ is the image of $P_1 * \cdots * P_r$ under the projection along the special simplex $e_1 , \ldots , e_r$. By \cite[Thm.\ 2.16]{BatyrevNill} the $h^*$-polynomial does not change under this projection.
\end{remark}

\section{Ideas for further investigation}\label{Ideas}

The attentive reader will have noticed that, although we have proven that the stringy $E$-function of a Gorenstein polytope is a polynomial, we did not say anything about its expected degree, which is twice the Calabi-Yau dimension $n$ by Conjecture~\ref{conjectureBatyrevNill} (2).   

The following lemma gives a formula for the coefficient of $(uv)^n$. Note that it equals the constant coefficient by the Poincar\'e duality formula.

\begin{lemma}\label{LemmaLeading}
Let $P$ be a Gorenstein polytope of index $r$. The constant coefficient of $E_{st}(P;u,v)$ equals the cardinality of the set of faces $F$ of $P$ satisfying the following conditions:
\begin{itemize}
\item[(1)] $\deg F = \deg \widetilde{S}(F,t)$ and $\deg F^* = \deg \widetilde{S}(F^*,t)$,
\item[(2)] $\textnormal{codeg}\, F + \textnormal{codeg}\, F^* = r$,
\item[(3)] $\dim F + 1 = 2 (r - \textnormal{codeg}\, F^*)$.
\end{itemize}
In particular, only odd dimensional faces $F$ contribute to $E_{st}(P;0,0)$.
\end{lemma}

\begin{proof}
In Corollary~\ref{proofofconjecture} we have written $E_{st}(P;u,v)$ as $\sum_{\emptyset \leq F\leq P} A_1(F) \, A_2(F)$, where
\begin{eqnarray*}
A_1(F) & = & \frac{(-u)^{\dim F + 1}\,\widetilde{S}(F,u^{-1}v)}{(uv)^{r - \text{codeg}\, F^{*}}} \\
A_2(F)& = &  \frac{\widetilde{S}(F^{*},uv)}{(uv)^{\text{codeg}\, F^{*}}}, 
\end{eqnarray*}
and both $A_1(F)$ and $A_2(F)$ are polynomials if $A_1(F)\neq 0$ and $A_2(F)\neq 0$. Let $F$ be a face for which both $A_1(F)$ and $A_2(F)$ have a nonzero constant coefficient. Since $A_1(F)$ is homogeneous of degree $\dim F + 1 - 2(r - \text{codeg}\, F^{*})$, we have condition (3) immediately, and hence $\dim F$ is odd. By Corollary~\ref{corollaryStilde} (2) we have that $\text{codeg}\, F^* \leq \text{subdeg}\, \widetilde{S}(F^*,t)$, but in order to have a nonzero constant coefficient in $A_2(F)$ we need equality, and hence $\deg F^* = \deg \widetilde{S}(F^*,t)$. 

Since $\widetilde{S}(F,t)\neq 0$, it follows from the reciprocity law for $\widetilde{S}$-polynomials that 
\[ \deg \widetilde{S}(F,t) \geq  \frac{\dim F + 1}{2}  \quad \text{and} \quad \text{subdeg}\, \widetilde{S}(F,t) \leq  \frac{\dim F + 1}{2}.\]
From condition (3), Theorem~\ref{maintheorem} and Corollary~\ref{corollaryStilde} (2) we get
\[ \frac{\dim F + 1}{2} = r - \text{codeg}\, F^* \leq \text{codeg}\, F \leq \text{subdeg}\, \widetilde{S}(F,t) \leq \frac{\dim F + 1}{2}. \]
It follows that we have equality everywhere and hence by the reciprocity law again, conditions (1) and (2) hold. Moreover, by Corollary~\ref{corollaryStilde} (1) we have $\widetilde{S}(F,t)\leq h_F^*(t)$, and Theorem~\ref{equality} tells us that $F$ is Gorenstein. Hence the leading coefficient of $h_F^*(t)$ is $1$. It follows from the above that $\widetilde{S}(F,t) = t^{(\dim F + 1) / 2}$ and hence $A_1(F) =1$. Similarly, $F^*$ is Gorenstein and hence the leading and constant coefficient of $\widetilde{S}(F^*,t)$ or $A_2(F)$ are $1$ as well. 
\end{proof}

In the case of an irreducible Gorenstein polytope, only $\emptyset$ and $P$ might contribute to the constant coefficient of the stringy $E$-function, so we have the following corollary.

\begin{corollary}
Let $P$ be an irreducible Gorenstein polytope. The constant coefficient of $E_{st}(P;u,v)$ belongs to $\{0,1,2\}$.
\end{corollary}

In Lemma~\ref{LemmaLeading} the empty face looks like an obvious candidate to contribute to the constant coefficient. Therefore an affirmative answer to the following questions would prove Conjecture~\ref{conjectureBatyrevNill} (2).

\begin{question}
Let $P$ be a Gorenstein polytope of index $r$ and dimension $d$.
\begin{itemize}
\item[(1)] Assume that $E_{st}(P;u,v) \neq 0$. Is it true that $\widetilde{S}(P^{\times},t)\neq 0$\,?
\item[(2)] Assume that $\widetilde{S}(P,t)\neq 0$. Is it true that $\deg \widetilde{S}(P,t) = \deg P$\,?
\end{itemize}
\end{question}

In \cite[Rem.\ 4.11 and 4.22]{BatyrevNill} the question was raised whether the leading coefficient of the stringy $E$-function is always a power of 2. We believe that this is true. Our hope was that the constant coefficient would behave multiplicatively with respect to the kind of joins of Theorem~\ref{equality}, but Example~\ref{HardExample} shows that this fails. However, in this example one has
\[ E_{st}(P;u,v) = E_{st}(P^{\times};u,v ) = 8 ,\]
\[ E_{st}(F;u,v) = E_{st}(G;u,v ) = 2, \quad  E_{st}(F^*;u,v) = E_{st}(G^*;u,v ) = 4.\] 
This leads us to ask the following question.

\begin{question}\label{leading}
Let $P$ be a Gorenstein polytope of index $r$ and let $F$ be a face of $P$ with $\text{codeg}\, F + \text{codeg}\, F^* = r$. Is it true that
\[   E_{st}(P;0,0)\,  E_{st}(P^{\times};0,0)  = E_{st}(F;0,0)\, E_{st}(G;0,0)\, E_{st}(F^*;0,0)\, E_{st}(G^*;0,0)\,?   \]
\end{question}

One might even ask for such a multiplicative behaviour for the whole stringy $E$-function. 

Finally, Example~\ref{BatyBoriExample} inspires us to ask the following question.

\begin{question}
Let $P$ be a reflexive polytope and $\Pi(P)=\{ P_1 , \ldots , P_r\} $ be a nef partition that splits as a direct sum
\[ \Pi(P)  =  \Pi(P_{I_1}) \oplus \Pi(P_{I_2}),\]
where $\{1,\ldots , r\}$ is the disjoint union of $I_1$ and $I_2$. Is
\[  E_{st}(\widetilde{P};u,v) = E_{st}(\widetilde{P_{I_1}} ; u,v ) \, E_{st}(\widetilde{P_{I_2}};u,v), \]
or does such multiplicative behaviour at least hold for the constant coefficient\,?
\end{question}

\vspace{3mm}

\footnotesize

\textsc{Benjamin Nill, Department of Mathematics, University of Georgia, Athens, GA 30602, USA} 

\textsc{Email:} \texttt{bnill@math.uga.edu}\\

\textsc{Jan Schepers, Department of Mathematics, K.U.Leuven, Celestijnenlaan 200B, B-3001 Leuven, Belgium}

\textsc{Email:} \texttt{janschepers1@gmail.com}


\begin{thebibliography}{DL}
\bibitem[Ba1]{Batyrev1} V.\ V.\ Batyrev, \emph{Dual polyhedra and mirror symmetry for Calabi-Yau hypersurfaces in toric varieties}, J.\ Algebraic Geom.\ \textbf{3} (1994), 493-535.

\bibitem[Ba2]{Batyrev} V.\ V.\ Batyrev, \emph{Stringy Hodge numbers
of varieties with Gorenstein canonical singularities}, Proc.\
Taniguchi Symposium 1997, In `Integrable Systems and Algebraic
Geometry, Kobe/Kyoto 1997', World Sci.\ Publ.\ (1999), 1-32.

\bibitem[Ba3]{BatyrevVirasoro} V.\ V.\ Batyrev, \emph{Stringy Hodge numbers and Virasoro algebra}, Math.\ Res.\ Lett.\ \textbf{7} (2000), 155-164.

\bibitem[BB1]{BatyrevBorisovCompleteIntersections} V.\ V.\ Batyrev and L.\ A.\ Borisov, \emph{On Calabi-Yau complete intersections in toric varieties}, Higher-dimensional complex varieties (Trento, 1994), Walter de Gruyter (1996), 39-65.

\bibitem[BB2]{BatyrevBorisov} V.\ V.\ Batyrev and L.\ A.\ Borisov, \emph{Mirror duality and string-theoretic
Hodge numbers}, Invent.\ Math.\ \textbf{126} (1996), 183-203.

\bibitem[BB3]{BatyrevBorisov2} V.\ V.\ Batyrev and L.\ A.\ Borisov, \emph{Dual cones and mirror symmetry for generalized Calabi-Yau manifolds}, Mirror symmetry II, AMS/IP Stud.\ Adv.\ Math.\ \textbf{1} (1997), 71-86.

\bibitem[BN]{BatyrevNill} V.\ V.\ Batyrev and B.\ Nill, \emph{Combinatorial aspects of mirror symmetry},
in `Integer points in polyhedra - geometry, number theory, representation theory, algebra, optimization, statistics', Contemp. Math. \textbf{452} (2008), 35-66.

\bibitem[Bo]{Borisov} L.\ A.\ Borisov, \emph{Towards the mirror symmetry for Calabi-Yau
complete intersections in Gorenstein toric Fano
varieties}, arXiv:alg-geom/9310001v1.

\bibitem[BM]{BorisovMavlyutov} L.\ A.\ Borisov and A.\ R.\ Mavlyutov, \emph{String cohomology of Calabi-Yau
hypersurfaces via mirror symmetry}, Adv.\ Math.\ \textbf{180} (2003), 355-390.

\bibitem[Br]{Braun} B.\ Braun, \emph{An Ehrhart series formula for reflexive polytopes}, Electr.\ J.\ Comb.\ \textbf{13} (2006), \#N15.


\bibitem[DN]{DoranNovoseltsev} C.\ F.\ Doran and A.\ Y.\ Novoseltsev, \emph{Closed form expressions for Hodge numbers
of complete intersection Calabi-Yau threefolds in toric varieties}, arXiv:0907.2701v1 [math.CO].

\bibitem[Ehr]{Ehrhart} E.\ Ehrhart, \emph{Polyn\^omes arithm\'etiques et m\'ethode des poly\`edres en combinatoire}, International Series of Numerical Mathematics, Vol.\ 35, Birkh\"auser Verlag (1977).

\bibitem[HRZ]{HenkRichterZiegler} M.\ Henk, J.\ Richter-Gebert and G.\ Ziegler, \emph{Basic properties of convex polytopes},
Chapter 16 of the second edition of the "CRC Handbook of Discrete and Computational
Geometry", edited by J.\ E.\ Goodman and J.\ O'Rourke. (2004), 355–382.

\bibitem[HT]{HenkTagami} M.\ Henk and M.\ Tagami, \emph{Lower bounds on the coefficients of Ehrhart polynomials}, European J.\ Combin.\ \textbf{30} (2009), 70-83.

\bibitem[Hi]{Hibi} T.\ Hibi, \emph{Dual polytopes of rational convex polytopes}, Combinatorica \textbf{12} (1992), 237-240.

\bibitem[Sch]{Schepers} J.\ Schepers, \emph{Stringy Hodge numbers of strictly canonical nondegenerate singularities}, to appear in J.\ Algebraic Geom.

\bibitem[BI]{normaliz} W.\ Bruns and B.\ Ichim, \emph{Normaliz 2.0}, available on\\ 
\verb1http://www.mathematik.uni-osnabrueck.de/normaliz/index.html1.


\bibitem[St1]{Stanley80} R.\ P.\ Stanley, \emph{Decompositions of rational convex polytopes}, Ann.\ Disc.\ Math.\ \textbf{6} (1980), 333-342.

\bibitem[St2]{Stanley87} R.\ P.\ Stanley, \emph{Generalized $h$-vectors, intersection cohomology of toric varieties, and
related results}, Commutative Algebra and Combinatorics (Kyoto, 1985), Adv.\ Stud.\ Pure Math.\ \textbf{11}
(1987), 187-213.

\bibitem[St3]{Stanley91} R.\ P.\ Stanley, \emph{On the Hilbert function of a graded Cohen-Macaulay domain}, J.\ Pure Appl.\ Algebra \textbf{73} (1991), 307-314.

\bibitem[St4]{Stanley92} R.\ P.\ Stanley, \emph{Subdivisions and local $h$-vectors},
J.\ Amer.\ Math.\ Soc.\ \textbf{5} (1992), 805-851.

\bibitem[St5]{Stanley97} R.\ P.\ Stanley, \emph{Enumerative Combinatorics, Volume 1}, Cambridge Studies in
Advanced Mathematics \textbf{49}, Cambridge University Press (1997).

\end{thebibliography}
\end{document}